\documentclass[12pt,reqno]{amsart}
\usepackage{amsmath, amsfonts, amssymb, amsthm, amscd, amsbsy, mathtools}
\usepackage{fancyhdr}
\usepackage[dvipsnames,svgnames,x11names]{xcolor}
\usepackage{graphicx}
\usepackage{geometry}
\usepackage[]{hyperref}
\usepackage{verbatim}
\usepackage{dsfont}
\usepackage{enumitem}
\usepackage{xparse}
\PassOptionsToPackage{nopatch=footnote}{microtype}
\usepackage{microtype}
\usepackage{fullpage}
\usepackage{setspace}
\usepackage{parskip} 
\usepackage{amsrefs}
\usepackage{comment}
\usepackage[commandnameprefix=ifneeded]{changes}

\numberwithin{equation}{section}
\providecommand{\U}[1]{\protect\rule{.1in}{.1in}}
\hypersetup{
colorlinks=true,
breaklinks=true,
urlcolor=NavyBlue,
linkcolor=Fuchsia,
bookmarksopen=false,
filecolor=black,
citecolor=red,
linkbordercolor=red
}
\allowdisplaybreaks
\newtheorem{theorem}{Theorem}[section]

\newtheorem{definition}{Definition}[section]
\newtheorem{proposition}[theorem]{Proposition}

\newtheorem{lemma}[theorem]{Lemma}
\newtheorem{remark}[]{Remark}
\newtheorem{example}[theorem]{Example}
\newtheorem{examples}[theorem]{Examples}
\newtheorem{foo}[theorem]{Remarks}

\newcommand{\norm}[1]{{\left\lVert{#1}\right\rVert}}

\newcommand{\Deltah}{\Delta_{\mathbb{H}}}
\newcommand{\mathleft}{\@fleqntrue\@mathmargin0pt}
\newcommand{\mathcenter}{\@fleqnfalse}
\newcommand{\disth}{d_{\mathbb{H}}}

\def\vint{\mathop{\mathchoice%
          {\setbox0\hbox{$\displaystyle\intop$}\kern 0.22\wd0%
           \vcenter{\hrule width 0.6\wd0}\kern -0.82\wd0}%
          {\setbox0\hbox{$\textstyle\intop$}\kern 0.2\wd0%
           \vcenter{\hrule width 0.6\wd0}\kern -0.8\wd0}%
          {\setbox0\hbox{$\scriptstyle\intop$}\kern 0.2\wd0%
           \vcenter{\hrule width 0.6\wd0}\kern -0.8\wd0}%
          {\setbox0\hbox{$\scriptscriptstyle\intop$}\kern 0.2\wd0%
           \vcenter{\hrule width 0.6\wd0}\kern -0.8\wd0}}%
          \mathopen{}\int}

\newcommand{\hn}{\mathbb{H}^n}
\newcommand{\bn}{\mathbb{B}^n}

\newcommand{\rn}{\mathbb{R}^n}
\newcommand{\sn}{\mathbb{S}^n}

\renewcommand{\leq}{\leqslant}

\renewcommand{\geq}{\geqslant}

\begin{document}

\title{Symmetry of Solutions to Fractional Semilinear Equations on Hyperbolic Spaces}
\author{Jianxiong Wang}
\address{Jianxiong Wang: Department of Mathematics\\Rutgers University\\Piscataway, NJ 08854, USA.}
\email{jiangxiong.wang@rutgers.edu}
\date{}
\begin{abstract}
We study a semilinear equation involving the fractional Laplacian on the hyperbolic space $\mathbb{H}^n$. Unlike in conformally compact Einstein manifolds, the fractional Laplacian on $\mathbb{H}^n$ does not enjoy conformal covariance. By employing Helgason-Fourier analysis, we explicitly derive the Green's function of the fractional Laplacian on $\mathbb{H}^n$ as well as its asymptotic behaviors. We then apply a direct method of moving planes to the integral form of the equation, and show that nonnegative weak solutions are symmetric. In addition, we extend several maximum principles to hyperbolic space.
\end{abstract}
\setlength\parindent{0pt}
\maketitle

\section{Introduction}

The fractional Laplacians arise in various fields of mathematics and physics, such as probability, finance, physics, and fluid dynamics. For instance, they arise naturally in the study of stochastic processes with jumps, and more precisely in L\'evy processes \cite{Levy}, which extend the concept of Brownian motion. Fractional order operators also appear in the study of conformal geometry and partial differential equations. In $\rn$, the fractional Laplacian is a nonlocal pseudo-differential operator, assuming the form
\[
(-\Delta)^s u(x) = c_{n,s} \, \text{P.V.} \int_{\rn} \frac{u(x) - u(y)}{|x-y|^{n+2s}} \, dy, \quad 0 < s < 1,
\]
where P.V. stands for the Cauchy principal value and $c_{n,s}$ is a normalization constant. Equivalently, $(-\Delta)^{\frac{\alpha}{2}}$ can be defined in terms of Fourier transform:
\begin{equation*}
    (-\Delta)^{s}u(x)=\mathcal{F}^{-1}[|\xi|^{2s} \mathcal{F}{u}(\xi)](x).
\end{equation*}

The non-locality of the fractional Laplacian makes it usually difficult to investigate. To circumvent this difficulty, Caffarelli and Silvestre \cite{CS07} introduced the extension method that reduced this nonlocal problem into a local one in higher dimensions. For a function $u : \rn \to \mathbb{R}$, we construct the extension $U : \rn \times [0, +\infty) \to \mathbb{R}, U = U (x, y)$, as the solution of the equation
\begin{equation}
\begin{cases}
\operatorname{div}(y^{1-2s} \nabla U(x,y)) = 0, & (x,y) \in \mathbb{R}^n \times (0,\infty), \\
U(x,0) = u(x), & x \in \mathbb{R}^n.
\end{cases}
\end{equation}

The fractional Laplacian is then recovered by:
\begin{equation}
(-\Delta)^s u(x) = - C_s \lim_{y \to 0^+} y^{1-2s} \frac{\partial U}{\partial y}(x,y),
\end{equation}
where $C_s$ is a normalization constant depending on $s$.

The fractional (as well as integer) powers of Laplacian that enjoy the conformal property can also be defined on Riemannian manifolds. Graham and Zworski \cite{GZ03} studied the connection between scattering matrices on conformally compact asymptotically Einstein manifolds and conformally invariant objects on their boundaries at infinity. Let $(X^{n+1}, g^+)$ be a conformally compact Einstein manifold with conformal infinity $(M, [\hat{g}])$. For any defining function $\rho$ of $M$, we can write $\bar{g} = \rho^2 g^+$, which extends to a metric on $\bar{X} = X \cup M$. Given $f \in C^\infty(M)$, consider the generalized eigenvalue problem
\begin{equation*}
  -\Delta_{g^+} u - s(n-s) u = 0 \quad \text{ in } X,
\end{equation*}
with the asymptotic expansion near $M$:
\begin{equation*}
  u = F \rho^{n-s} + G \rho^s, \quad F, G \in C^\infty(\bar{X}), \quad F|_{\rho=0} = f.
\end{equation*}
The scattering operator $S(\lambda)$ is defined by $S(\lambda) F = G$.
In \cite{CG11}, Chang and Gonz\'alez showed that the fractional order operators defined via the scattering operator can be realized as the Dirichlet-to-Neumann map of a degenerate elliptic equation on $X$, generalizing the extension method of Caffarelli-Silvestre to the setting of conformally compact Einstein manifolds. It is also worth mentioning that the fractional GJMS operators on hyperbolic spaces were explicitly calculated by Lu and Yang in \cite{LY23}. They also pointed out in \cite{FLY2} that these fractional GJMS operators are not conformally to the fractional Laplacians on the upper half space $\mathbb{R}^n_+$ nor on the unit ball $B^n$ in $\rn$, which one may expect for the integer cases. 

Although the fractional order operators on general manifolds are rather sophisticated, see e.g. \cites{GZ03,CG11} and the references therein, there have been numerous accomplishments in the past decades when one considers the integer powers of Laplacians, which we will discuss in the following briefly. In the celebrated work of Gidas-Ni-Nirenberg \cite{GNN1}, they considered the following boundary value problem on the ball $B_R(0)\subset{\rn}$.
\begin{equation*}
    \begin{cases}
        -\Delta u = f(u) &\text{ in } B_R(0)\\
         u=0 &\text{ on } \partial B_R(0),
    \end{cases}
\end{equation*}
where $f$ is of class $C^1$. They proved that any positive solution $u$ in $C^2(\overline{B_R(0)})$ is radially symmetric and decreasing.
Their approach is the so-called moving plane method, which was initiated by Alexandrov \cite{Alexandrov} in the 1950s, and further developed by Serrin \cite{serrin}.
Later, Gidas-Ni-Nirenberg \cite{GNN2} studied the equation $-\Delta u = f(u)$ in the entire space $\rn\setminus\{0,\infty\}, n\geq 3$
with two singularities located at the origin and infinity:
\begin{align}
  u(x)\to+\infty \quad &\text{ as } \quad x\to0. \nonumber\\
  |x|^{n-2}u(x)\to+\infty \quad &\text{ as } \quad x\to\infty.
\end{align}
They showed that the solution is radially symmetric about the origin and decreasing.
Subsequently, Caffarelli-Gidas-Spruck \cite{CGS} considered this equation in a punctured ball when $f$ has critical growth.
To be more precise, they require the nonlinearity $f(t)$ to be a locally nondecreasing Lipschitz function with $f(0)=0$, and to satisfy the following condition:
for sufficiently large $t$, the function $t^{-\frac{n+2}{n-2}}f(t)$ is nonincreasing and $f(t)\geq ct^p$ for some $p\geq \frac{n}{n-2}$.
They showed that the solution $u$ is radially symmetric and decreasing.

The moving plane method was later extended to the study of higher order equations on $\mathbb{R}^n$.
The main challenge here is that the moving plane method heavily depends on the maximum principle, which does not always hold for higher order operators.
To overcome this obstacle, Chen, Li and Ou \cites{CLO,ChenLiOu2} developed a powerful moving plane method for integral equations, obtaining the symmetry of solutions for higher order and even fractional order equations. More precisely, they proved that for the following equation.
$$
  (- \Delta )^{\frac{\alpha}{2}} u = u^{\frac{n+\alpha}{n - \alpha}}
$$
on $\mathbb{R}^n$, every positive regular solution (i.e. locally $L^{\frac{2n}{n-\alpha}}$ solutions) $u$
is radially symmetric and decreasing about some point $x_0$ and therefore assumes the form (up to some dilations):
$$
  u(x) = \frac{1}{\left( a + b |x - x_0|^2 \right)^{\frac{n-\alpha}{2}}}.
$$
Their work completely classifies all the critical points of the functional corresponding to the Hardy-Littlewood-Sobolev
inequalities of order $\alpha$, whose sharp constant was previously obtained by Lieb \cite{Lieb1}.

The method of moving planes is a powerful tool for proving symmetry and monotonicity properties of solutions to partial differential equations. It was first introduced by Alexandrov \cite{Alexandrov} in the context of geometric problems and later adapted by Serrin, Gidas, Ni, and Nirenberg \cites{K-P1, K-P2, CGS} to study symmetry properties of solutions to elliptic PDEs in Euclidean spaces. The method has since been extended to various settings, including Riemannian manifolds. In recent years, a direct method of moving planes was developed by Chen, Li and Li \cite{CLL} to deal with fractional Laplacians and other nonlocal operators. This method has been successfully applied to various problems involving fractional Laplacians in $\rn$ and other settings (see, e.g., \cites{ChenWu-ANS, ChenHuMa-ANS, LiY, Liao-ANS, LiuM-ANS, GuoMaZhang-ANS}).

In the present paper, we aim to study positive solutions to the following fractional order equations on $\hn$.
\begin{equation}
(-\Delta_{\hn})^s u = u^p(x), \quad x \in \hn.
\end{equation}
Here $0 < s < 1$ and $p > 1$. We start by reviewing some work in the Euclidean space. In \cite{BCPS}, among other results, the authors considered the properties of the positive solutions for  
\begin{equation*}
  (-\Delta)^s u = u^p, \quad x \in \rn,
\end{equation*}
and first used the above extension method to reduce the nonlocal problem into a local one for $U(x,y)$ in one higher dimensional half space $\rn \times [0, \infty)$, then applied the method of moving planes to show the symmetry of $U(x, y)$ in $x$, and hence derived the non-existence in the subcritical case.

Later, Chen, Li and Li \cite{CLL} developed a direct method of moving planes for the fractional Laplacian in $\rn$, without using the extension method. They established several maximum principles for antisymmetric functions and applied the method of moving planes to obtain symmetry and nonexistence results for positive solutions of semilinear equations involving the fractional Laplacian in $\rn$. 
In \cite{CLZ}, Chen, Li and Zhang introduced another direct method - the method of moving spheres for the fractional Laplacian, which is more powerful than the method of moving planes. The method of moving spheres can be used to capture the solutions directly rather than going through the usual procedure of proving radial symmetry of solutions and then classifying radial solutions. It is worth noticing that both their approaches relies on the Kelvin transform. However, on hyperbolic spaces, the Kelvin transform for the fractional Laplacian is not available due to the lack of conformal invariance. To overcome this difficulty, we will make use of the asymptotic behavior of solutions and the Hardy-Littlewood-Sobolev inequality on hyperbolic spaces to start the moving plane process.

We point out here that the moving plane method on $\mathbb{H}^n$ was first seen in the work of Kumaresan-Prajapat \cites{K-P1,K-P2}, where they established the analog result of Gidas-Ni-Nireneberg type, as well as solved an overdetermined problem on $\mathbb{H}^n$, which is an analogue of the problem on $\mathbb{R}^n$ initiated by Serrin \cite{serrin} (see also \cite{HLZ}, \cite{L-Z2}). 
The hyperbolic space serves as one of the most important models of Riemannian manifolds with constant curvature. For definitions and basic properties of hyperbolic spaces, we refer to Section \ref{sec:prelim}. The axial symmetry of solutions to a class of integral equations on half spaces was  studied by Lu-Zhu in \cite{L-Z}.  More recently, the overdetermined problem of fractional order equations on hyperbolic spaces was studied by Li-Lu-Wang \cite{LLW2}. The concept of the moving plane method in \cites{K-P1,K-P2} was further developed in the work of Almeida-Ge \cite{ADG2} and Almeida-Damascelli-Ge \cite{ADG1},
where they took advantage of the foliation structure of $\mathbb{H}^n$ (for details see Section \ref{sec:prelim}). 
Li, Lu and Yang \cite{LLY} obtained a higher order symmetry result of the Brezis-Nirenberg problem on the hyperbolic spaces \cite{LLY}. Recently. with Li and Lu, the author studied general higher order equations \cite{LLW} by using the moving plane method, and later obtained the classification result using the moving sphere method \cite{LLW25} on the hyperbolic space. For symmetry and existence results of semilinear PDEs of second order on hyperbolic space, see Mancini and Sandeep \cite{ManciniSandeep1}, and \cite{Dutta-Sandepp} on a quasilienar equation involving the $p$-Laplacian on $\hn$.

The definition of the fractional Laplacian on hyperbolic spaces via the singular integral representation was constructed in \cites{BGS15}.
\begin{definition}
Let $n \geq 2$ and $0 < s < 1$. The fractional Laplacian on $\mathbb{H}^n$ is defined by
\begin{equation}
  (-\Delta_{\mathbb{H}^n})^s u(x) = c_{n,s} \, \text{P.V.} \int_{\mathbb{H}^n} (u(x) - u(\xi)) \, \mathcal{K}_{n,s}(d(x,\xi)) \, d\xi
\end{equation}
with the kernel $\mathcal{K}_{n,s}$ given by
\begin{equation}\label{kernel}
  \mathcal{K}_{n,s}(\rho) = C_1 \left( -\frac{\partial_\rho}{\sinh \rho} \right)^{\frac{n-1}{2}} \left( \rho^{-\frac{1+2s}{2}} K_{\frac{1+2s}{2}} \left( \frac{n-1}{2} \rho \right) \right)
\end{equation}
when $n \geq 3$ is odd and
\[
\mathcal{K}_{n,s}(\rho) = C_1 \int_{\rho}^{\infty} \frac{\sinh r}{\sqrt{\pi}\sqrt{\cosh r - \cosh\rho}} 
\left( -\frac{\partial_r}{\sinh r} \right)^{n/2}
\left( r^{-\frac{1+2s}{2}} K_{\frac{1+2s}{2}} \left( \tfrac{n-1}{2}r \right) \right) dr
\]
when $n \geq 2$ is even, where
\[
c_{n,s} = 2 \, \frac{4\sqrt{2} \,\Gamma(n+s)}{3\,\Gamma(n/2)\,\Gamma(-s)}, \quad
C_1 = \frac{1}{2^{n-2+2s}\,\Gamma\!\left( \frac{n-1}{2} \right) \Gamma\!\left( \frac{1+2s}{2} \right)},
\]
and $K_\nu$ is the modified Bessel function of the second kind.
\end{definition}
We consider the following classic semilinear equation with power nonlinearity involving the fractional Laplacian:
\begin{equation}\label{main_equation_1}
(-\Delta_{\hn})^s u = u^p(x), \quad x \in \hn,
\end{equation}
where $1 < p \leq p^*=\frac{n+2s}{n-2s}$.
We say that a nonnegative function $u \in H^s(\hn)$ is a weak solution of \eqref{main_equation_1} if for any test function $\varphi \in C_c^\infty(\hn)$, there holds
\begin{align*}
  \int_{\mathbb{H}^n} u(x) (-\Delta_{\mathbb{H}^n})^s \varphi(x) dV_x &= \int_{\mathbb{H}^n} u(x)^p \varphi(x) dV_x.
\end{align*}

We point out that the existence of nonnegative weak solutions has been showed recently by Bruno and Papageorgiou \cite{BP25} for subcritical exponents. To be more precise, they considered the following more general problem on $\hn$:
\begin{equation*}
    (-\Delta_{\hn})^s u - \lambda ^s u  - u^p = 0,
\end{equation*}
where $\lambda\in(0,1)$ and $1<p<p^*$, and showed that there exists at least one nontrivial nonnegative weak solution in $H^s(\hn)$. They also established the boundedness and regularity of the weak solutions, among many other results.

\begin{theorem}\label{thm1}
Let $n \geq 2$, $0 < s < 1$. Suppose that $u\in H^s(\hn)$ is a nonnegative weak solution of \eqref{main_equation_1}. 
Then there exists a point $P \in \mathbb{H}^n$ such that $u$ is radially symmetric and nondecreasing about $P$.
\end{theorem}

Another key novelty of this paper is that we take advantage of the Helgason-Fourier analysis on hyperbolic spaces to establish the equivalence between the differential form \eqref{main_equation_1} and the integral form
\begin{equation}\label{main_equation_int}
  u(x) = \int_{\hn} G_s(x,y) u^p(y) \, dV_y,
\end{equation}
where $G_{\hn}(x,y)$ is the kernel of the fractional Laplacian on $\hn$ given by 
\begin{equation}\label{Green's fun}
  G_s(\rho) = \int_{-\infty}^\infty \int_{\mathbb{S}^{n-1}} \left(\lambda^2 + \tfrac{(n-1)^2}{4}\right)^{-s} \bar{h}_{\lambda,\theta}(x) h_{\lambda,\theta}(y)| c(\lambda)|^{-2}d\theta d\lambda.
\end{equation}
This equivalence plays an essential role in applying the moving plane method to prove Theorem \ref{thm1}.
The Helgason-Fourier analysis on general symmetric spaces was first developed by Helgason \cites{H-F1,H-F2}. It has been effectively used in harmonic analysis and PDEs. In addition, the Helgason-Fourier analysis plays a key role in the establishment of higher order Hardy-Sobolev-Maz'ya's inequalities and Green's function of GJMS operators, which was studied by Lu and Yang \cites{LY1,LY2, LY3} and Flynn, Lu, Yang \cite{FLY1}. Using Fourier transform, we are able to derive the precise expression of the Green's function $G_s(\rho)$ and study its asymptotic behaviors near zero and infinity, which are the key ingredients in the moving plane method. For heat kernel and Green function estimates on general noncompact symmetric spaces, we refer to \cites{Anker-Ji, Yosida}, among many others.

\begin{theorem} \label{thm-asymptotics}
For $s\in(0,1)$, there holds that
\begin{itemize}
\item For $n\geq 3$ odd,
\begin{equation}\label{G1}
  G_s(\rho)=\alpha_\gamma \left(\frac{\partial_\rho}{\sinh\rho}\right)^\frac{n-1}{2}\rho^{-\frac{1}{2}+s}K_{-\frac{1}{2}+s}\left(\tfrac{n-1}{2} \rho\right),\end{equation}

\item For $n\geq 2$ even,
\begin{equation}\label{G2}
G_s(\rho)=\alpha_\gamma
\int_\rho^\infty\frac{\sinh r}{\sqrt{\cosh r-\cosh\rho}}
\left(\frac{\partial_r}{\sinh r}\right)^\frac{n}{2}\left[r^{-\frac{1}{2}+s}K_{-\frac{1}{2}+s}\left(\tfrac{n-1}{2} r\right)\right]\, dr.
\end{equation}
\end{itemize}
Here $K_{-\frac{1}{2}+s}$ is the solution to the modified Bessel equation given in Lemma \ref{lemma-Bessel}.

Additionally, $G_s(\rho)$ has the asymptotic behavior:
\begin{itemize}
\item As $\rho\to 0$,
  \begin{equation}
    \label{asymptotics-zero}G_s(\rho)\sim \rho^{-n+2s}.
  \end{equation}

\item As $\rho\to\infty$,
  \begin{equation}\label{asymptotics-infty}
    G_s(\rho)\sim \rho^{-1+s}e^{-(n-1)\rho}.
  \end{equation}
\end{itemize}
\end{theorem}

The organization of this paper is as follows. In Section \ref{sec:prelim}, we provide some preliminaries on hyperbolic space $\hn$. Section \ref{sec:frac_lap} collects some important information of the fractional Laplacians on $\hn$ and their key properties. The precise expression and asymptotic behaviors of Green's function are provided. Section \ref{sec:proof} is dedicated to the proof of the main theorem. In Section \ref{sec:MP}, we extend several maximum principles to $\hn$.

\section{Preliminaries}
\label{sec:prelim}

\subsection{Models of hyperbolic spaces}
\hfill

The hyperbolic space $\hn$ $(n\geq 2)$ is a complete, simply connected Riemannian manifold with constant sectional curvature $-1$.
There are several analytic models of hyperbolic spaces, all of which are equivalent.
Among them, we describe two models here.

\begin{itemize}
    \item \textit{The Half-space model}:
    It is given by $\mathbb{R}^{n-1}\times \mathbb{R}^+=\{(x_1,\cdots,x_{n-1},x_n):x_n>0\}$,
    equipped with the Riemannian metric
    $$ds^2=\frac{dx_1^2+\cdots+dx_n^2}{x_n^2}.$$
    The hyperbolic volume element is $dV = \frac{dx}{x_n^n}$, where $dx$ is the Lebesgue measure on $\mathbb{R}^n$. The hyperbolic gradient is $\nabla_\mathbb{H} = x_n \nabla$ and the Laplace-Beltrami operator on $\mathbb{H}^n$ is given by
$$
  \Delta_{\mathbb{H}^n} = x_n^2 \Delta - (n-2) x_n \frac{\partial}{\partial x_n},
$$
where $\Delta$ is the usual Laplacian on $\mathbb{R}^n$.

    \item \textit{The Poincar\'e ball model}:
    It is given by the open unit ball $\mathbb{B}^n=\{x=(x_1,\cdots,x_n):x_1^2+\cdots+x_n^2<1\}\in\mathbb{R}^n$ equipped with the Poincar\'e metric
    $$ds^2=\frac{4\left(dx_1^2+\cdots+dx_n^2\right)}{\left(1-|x|^2\right)^2}.$$
    The distance from $x\in\mathbb{B}^n$ to the origin is $\rho(x,0)=\log\frac{1+|x|}{1-|x|}$.
    The hyperbolic volume element is $dV=\left(\frac{2}{1-|x|^n}\right)^ndx$.
    The hyperbolic gradient is $\nabla_{\mathbb{H}}=\frac{1-|x|^2}{2}\nabla$ and the associated Laplace-Beltrami operator is given by
    $$\Delta_{\mathbb{H}^n}=\frac{1-|x|^2}{4}\left((1-|x|^2)\Delta+2(n-2)\sum_{i=1}^nx_i\frac{\partial}{\partial x_i}\right).$$
    \item \textit{The hyperboloid model}:
    It is given by the upper sheet of the two-sheeted hyperboloid in $\mathbb{R}^{n,1}$:
    $$\mathbb{H}^n=\{x=(x_0,x_1,\cdots,x_n)\in\mathbb{R}^{n,1}:x_0^2-x_1^2-\cdots-x_n^2=1,x_0>0\},$$
    equipped with the metric induced from the Lorentzian metric on $\mathbb{R}^{n,1}$:
    $$  ds^2=-dx_0^2+dx_1^2+\cdots+dx_n^2.$$  
\end{itemize}

\subsection{Hardy-Littlewood-Sobolev inequality}

The Hardy-Littlewood-Sobolev inequality inequality on hyperbolic ball model $\bn$ is equivalent to the one on the hyperbolic upper half spaces model.
It was first proved on half spaces by Beckner \cite{Beckner}, and then on Poincar\'e ball by Lu and Yang \cite{LY1}.

\textbf{Theorem.} \textit{Let $0<\lambda<n$ and $p=\frac{2n}{2n-\lambda}$. Then for $f,g\in L^p(\bn)$,}
\begin{equation}\label{HLS}
    \left|\int_{\bn}\int_{\bn}\frac{f(x)g(y)}{(2\sinh (\frac{\rho(x,y)}{2}))^\lambda}dV_xdV_y\right|\leq C_{n,\lambda}\|f\|_p\|g\|_p,
\end{equation}
\textit{where $\rho(x,y)$ denotes the geodesic distance between $x$ and $y$ and }
\begin{equation*}
    C_{n,\lambda}=\pi^{\lambda/2}\frac{\Gamma(\frac{n}{2}-\frac{\lambda}{2})}{\Gamma(n-\frac{\lambda}{2})}\left(\frac{\Gamma(\frac{n}{2})}{\Gamma(n)}\right)^{-1+\frac{\lambda}{n}}
\end{equation*}
\textit{is the best constant for the classical Hardy-Littlewood-Sobolev constant on $\mathbb{R}^n$.
Furthermore, the constant $C_{n,\lambda}$ is sharp and there is no nonzero extremal function for the inequality (\ref{HLS}).}

Sometimes it is more convenient to rewrite the inequality (\ref{HLS}) as
\begin{equation}\label{HLS-2}
    \|I_\alpha f\|_{L^{q}(\hn)} \lesssim \|f\|_{L^{p}(\hn)}.
\end{equation}
where $\frac{1}{q}=\frac{1}{p}-\frac{\alpha}{n}, 1<p<\frac{n}{\alpha}$, and $I_\alpha: L^p(\hn)\to L^q(\hn)$ is given by
\[I_\alpha f(x)=\int_{\mathbb H^n} \frac{f(y)}{\sinh^{n-\alpha}\frac{\rho(x,y)}{2}}\,dV_g(y).\]

\subsection{Sobolev inequality for fractional Laplacians on \texorpdfstring{$\hn$}{}}
\hfill

We collect some known Sobolev inequalities for fractional Laplacians on hyperbolic spaces; see e.g. \cites{BGS15}. Denote by $H^s(\mathbb H^n)$ the fractional Sobolev space on $\mathbb H^n$ defined as the completion of $C_c^\infty(\mathbb H^n)$ under the norm
\[\|u\|_{H^s(\mathbb H^n)}=\Big(\|u\|_{L^2(\mathbb H^n)}^2+\mathcal E_s(u,u)\Big)^{1/2},\]
where
\[\mathcal E_s(u,u):=\langle u,(-\Delta_{\mathbb H^n})^s u\rangle_{L^2(\mathbb H^n)}.\]
\begin{theorem}
There exists a constant $C=C(n,s)>0$ such that for every $u\in C_c^\infty(\mathbb H^n)$,
\begin{equation}\label{frac-sob-hn}
\Bigl(\int_{\mathbb H^n} |u|^{2^*_s}\,dV_g\Bigr)^{2/2^*_s}
\le C\,\mathcal E_s(u,u).
\end{equation}
\end{theorem}

\begin{theorem}
Let $H^s_{\mathrm{rad}}(\mathbb H^n)\subset H^s(\mathbb H^n)$ denote the subspace of radial functions, then one has the compact embedding:
\[
H^s_{\mathrm{rad}}(\mathbb H^n)\hookrightarrow L^q(\mathbb H^n)\qquad\text{for } 1\le q<\frac{n+2s}{n-2s}.
\]
\end{theorem}
For a general survey and the manifold setting, one may be referred to Hebey \cite{Hebey00}. 

\subsection{The Helgason-Fourier transform on hyperbolic spaces}
\hfill

Denote $$e_{\lambda,\zeta}(x)=\left(\frac{\sqrt{1-|x|^2}}{|x-\zeta|}\right)^{n-1+i\lambda}, \; x\in\mathbb{B}^n, \lambda\in\mathbb{R}, \zeta\in\mathbb{S}^{n-1}$$
the generalized eigenfunctions of the Laplace-Beltrami operator $\Delta_{\hn}$ on the ball model $\bn$ with eigenvalue $-(\lambda^2+\frac{(n-1)^2}{4})$. Equivalently, the eigenfunctions on the hyperboloid model is
\[h_{\lambda,\theta}=[x, (1, \theta)]^{i\lambda-(n-1)/2}\]
and satisfies $\Delta_{\hn} h_{\lambda,\theta} = -(\lambda^2 + \frac{(n-1)^2}{4}) h_{\lambda,\theta}$.
The Fourier transform of a function $f$ on $\mathbb{H}^n$ is defined as
$$\hat{f}(\lambda,\zeta)=\int_{\bn}f(x)e_{-\lambda,\zeta}(x)dV,$$
provided the integral exists. Moreover, the following inversion formula holds for $f\in C^\infty_0(\bn)$:
$$f(x)=D_n\int_{-\infty}^\infty\int_{\mathbb{S}^{n-1}}\hat{f}(\lambda,\zeta)e_{\lambda,\zeta}(x)|c(\lambda)|^{-2}d\lambda d\sigma,$$
where $D_n=(2^{3-n}\pi|\mathbb{S}^{n-1}|)^{-1}$ and $c(\lambda)$ is the Harish-Chandra's $c$-function given by
$$c(\lambda)=\frac{2^{n-1-i\lambda}\Gamma(n/2)\Gamma(i\lambda)}{\Gamma(\frac{n-1+i\lambda}{2})\Gamma(\frac{1+i\lambda}{2})}.$$
There also holds the Plancherel formula:
$$\int_{\bn}|f(x)|^2dV=D_n\int_{-\infty}^\infty\int_{\mathbb{S}^{n-1}}\hat{f}(\lambda,\zeta)|c(\lambda)|^{-2}d\lambda d\sigma.$$
For complete details, we refer to \cites{H-F1,H-F2}.

\subsection{Foliations of hyperbolic spaces}\label{sec:hyperbolic}
\hfill 

A foliation is an equivalence relation on a manifold, with the equivalence classes being connected, injectively  submanifolds, all of the same dimension.
Let $\mathbb{R}^{n,1}=(\mathbb{R}^{n+1},\cdot)$, where $\cdot$ is Lorentzian inner product defined by $x\cdot y=-x_0y_0+x_1y_1+\cdots+x_ny_n$. The hyperboloid model
of $\hn$ is the submanifold $\{x\in\mathbb{R}^{n,1}: x\cdot x=-1, x_0>0\}$. A particular directional foliation can be obtained by choosing any $x_i$ direction, $i=1,\cdots,n$.
Without loss of generality, we may choose $x_1$ direction. Denote $\mathbb{R}^{n,1}=\mathbb{R}^{1,1}\times\mathbb{R}^{n-1}$, where $(x_0,x_1)\in \mathbb{R}^{1,1}$. We define
$A_t=\tilde{A}_t\otimes Id_{\mathbb{R}^{n-1}}$, where $\tilde{A}_t$ is the hyperbolic rotation (also called a boost) in $\mathbb{R}^{1,1}$,
$$\tilde{A}_t=\begin{pmatrix}
    \cosh t & \sinh t \\
    \sinh t & \cosh t
\end{pmatrix}.$$
Let $U=\hn\cap\{x_1=0\}$ and $U_t=A_t(U)$, $\hn$ is then foliated by $U_t$ and $\hn=\bigcup_{t\in\rn}U_t$.
The reflection $I$ is an isometry such that $I^2=Id$ and $I$ fixes the hypersurface $U$, by $I(x_0,x_1,x_2,\cdots,x_n)=(x_0,-x_1,x_2,\cdots,x_n)$.
Moreover, the reflection w.r.t. $U_t$ is defined as $I_t=A_t\circ I\circ A_{-t}$, and $U_t$ is fixed by $I_t$.

\subsection{Bessel equations}
\begin{lemma}[\cite{AS64}, or Lemma 2.3 of \cite{BGS15}]\label{lemma-Bessel}
The solution of the ODE
\begin{equation*}
\label{Bessel1}\partial_{ss} \varphi+ \frac{\alpha}{s}\, \partial_s \varphi -\varphi = 0. \end{equation*}
may be written as $\varphi(s)=s^\nu \psi(s)$, for $\alpha=1-2\nu$, where $\psi$ solves the is the well known Bessel equation
\begin{equation}\label{Bessel2}
s^2\psi''+s\psi'-(s^2+\nu^2)\psi=0.
\end{equation}
In addition, \eqref{Bessel2} has two linearly independent solutions, $I_\nu,K_\nu$, which are the modified Bessel functions; their asymptotic behavior is given precisely by
\begin{align*}
I_\nu(s)&\sim \frac{1}{\Gamma(\nu+1)}\left(\frac{s}{2}\right)^\nu\left( 1+\frac{s^2}{4(\nu+1)}+\frac{s^4}{32(\nu+1)(\nu+2)}+\ldots\right),\\
K_\nu(s)&\sim \frac{\Gamma(\nu)}{2}\left(\frac{2}{s}\right)^{\nu}
\left( 1+\frac{s^2}{4(1-\nu)}+\frac{s^4}{32(1-\nu)(2-\nu)}+\ldots\right)
\\&\quad+\frac{\Gamma(-\nu)}{2}\left(\frac{s}{2}\right)^\nu\left( 1+\frac{s^2}{4(\nu+1)}+\frac{s^4}{32(\nu+1)(\nu+2)}+\ldots\right),
\end{align*}

for $s\to 0^+$, $\nu\not\in\mathbb Z$. And when $s\to +\infty$,
\begin{align*}\label{asymptotic2}
    I_\nu(s)\sim \frac{1}{\sqrt{2\pi s}}e^s\left(1-\frac{4\nu^2-1}{8s}+\frac{(4\nu^2-1)(4\nu^2-9)}{2!(8s)^2}-\ldots \right),\\
  K_\nu(s)\sim \sqrt{\frac{\pi}{2s}}e^{-s}\left(1+\frac{4\nu^2-1}{8s}+\frac{(4\nu^2-1)(4\nu^2-9)}{2!(8s)^2}+\ldots \right).
\end{align*}
\end{lemma}

\section{Fractional Laplacian on hyperbolic spaces}
\label{sec:frac_lap}
\hfill

The Caffarelli--Silvestre extension method can be directly applied to asymptotic hyperbolic Einstein manifolds $(X,g^+)$ with boundary $M$. In the case $M = \rn$ and $X = \mathbb{R}^{n+1}$ endowed with the hyperbolic metric $g_{\mathbb{H}}=\frac{dy^2+|dx|^2}{y^2}$, the scattering operator coincides with the extension problem for the fractional Laplacian when $s\in (0, 1)$. This is firstly observed by Chang and Gonz\'alez \cite{CG11}.
Consider the upper half-space model of $\mathbb H^{n+1}$:
\[
X = \mathbb R^n_x \times (0,\infty)_y, \quad g^+ = \frac{dy^2 + dx^2}{y^2}.
\]

The key result is that, in the case $M = \rn$ with $\hat{g}=|dx|^2$, $X = \mathbb{R}^{n+1}$ with coordinates $x \in \rn, y > 0$, endowed the hyperbolic metric $g_{\mathbb{H}} = \frac{dy^2+|dx|^2}{y^2}$, the scattering operator is nothing but the Caffarelli-Silvestre extension problem for the fractional Laplacian when $\gamma\in (0, 1)$.
For $\lambda \in \mathbb C$, consider
\[
(-\Delta_{g^+} - \lambda(n-\lambda)) U(x,y) = 0, \quad U \sim y^{n-\lambda} F + y^\lambda G \text{ as } y\to 0.
\]
The scattering operator $S(\lambda)$ is defined by $S(\lambda) F = G$, and the conformally covariant fractional powers of the Laplacian is defined as 
\begin{align}
  P_\gamma[\hat{g}, g^+] &= d_\gamma S\left(\frac{n}{2} + \gamma\right), \quad d_\gamma = 2^{2\gamma} \frac{\Gamma(\gamma)}{\Gamma(-\gamma)}. 
\end{align}
The operators $P_\gamma[\hat{g}, g^+]$ satisfy an important conformal covariance property \cite{GZ03}:
if $\hat{g}_w = w^{\frac{4}{n-2\gamma}} \hat{g}$, then
\[P_\gamma[\hat{g}_w, g^+] \phi = w^{-\frac{n+2\gamma}{n-2\gamma}} P_\gamma[\hat{g}, g^+] (w \phi).\]

Banica--Gonz\'alez--S\'aez \cite{BGS15} showed that for certain complete, noncompact manifolds $X$, the fractional Laplacian on $M$ can be defined via an extension problem on $X\times\rn_+$. 
In particular, they provided an explicit singular integral kernel representation of the fractional Laplacian on hyperbolic spaces $\hn$, using the Helgason-Fourier transform \cites{H-F1,H-F2}. For $u \in C_c^\infty(\mathbb H^n)$, the fractional Laplacian is defined by
\[(-\Delta_{\mathbb H^n})^s u(x) = \int_{-\infty}^\infty \int_{\mathbb S^{n-1}} \left(\lambda^2 + \tfrac{(n-1)^2}{4}\right)^s \, \hat{u}(\lambda, \theta) \, h_{\lambda,\theta}(x) \, \frac{d\theta d\lambda}{|c(\lambda)|^2},
\]
where $h_{\lambda,\theta} = [x,(1,\theta)]^{\,i\lambda-\frac{n-1}{2}}$ are the generalized eigenfunctions of the Laplace-Beltrami operator on $\mathbb H^n$ (hyperboloid model) which satisfy
\[\Delta_{\mathbb H^n} h_{\lambda,\theta} = -\left(\lambda^2 + \tfrac{(n-1)^2}{4}\right) h_{\lambda,\theta},\]
and $c(\lambda)$ is the Harish-Chandra coefficient.
This produces the singular integral kernel representation:
\[
(-\Delta_{\mathbb H^n})^s u(x) = C_{n,s} \; \mathrm{P.V.} \int_{\mathbb H^n} \frac{u(x)-u(y)}{(\sinh(d_{\mathbb H^n}(x,y)/2))^{n+2s}} \, dV_y.
\]
We point out here that the above two definitions of fractional Laplacians are not equivalent. Although the later is also defined via an extension problem, the manifold $\hn\times (0,\infty)$ is not conformally compact Einstein. In fact, the former operator (Chang-Gonz\'alez) lives on the conformal infinite of $X$ while the latter (Banica--Gonz\'alez--S\'aez) is defined on the bulk manifold $X$ itself. As a consequence, the fractional Laplacian we treat here does not coincide with the scattering operator on hyperbolic spaces, which is not conformally covariant. 

The following asymptotic estimates of the kernel $\mathcal{K}_{n,s}(\rho)$ were obtain in \cite{BGS15}.
\begin{proposition}[\cite{BGS15}]
There exist constants $ c, C > 0 $ such that
\[
c \, \rho^{-\frac{1+2\rho}{2}} (\sinh \rho)^{-\frac{n-1}{2}} K_{\frac{n+s\rho}{2}} \!\left( \frac{n-1}{2} \rho \right)\leq K_{n,s,\rho}(\rho) \leq C \, \rho^{-\frac{1+2\rho}{2}} (\sinh \rho)^{-\frac{n-1}{2}} K_{\frac{n+2\rho}{2}} \!\left( \frac{n-1}{2} \rho \right)
\]
In particular,
\[
\mathcal{K}_{n,s}(\rho) \sim \rho^{-n-2s}
\]
as $ \rho \to 0^+ $, and
\[
\mathcal{K}_{n,s}(\rho) \sim \rho^{-1 - s} e^{-(n-1)\rho}
\]
as $ \rho \to \infty $.
\end{proposition}

\begin{proposition}[\cite{BGS15}]
  The kernel $\mathcal{K}_{n,s}(\rho)$ is positive.
\end{proposition}

We next discuss the kernel $G_s$ in the integral equation \eqref{main_equation_int}. Define 
\begin{equation}
  L_\lambda(x,y) = \int_{\sn} \bar{h}_{\lambda,\theta}(x)h_{\lambda,\theta}(y) \, d\theta.
\end{equation}
We note the following explicit formulae of $L_s(x,y)$ given by Banica \cite{Banica07}.
\begin{lemma}[\cite{Banica07}]\label{L_representation}
  \begin{equation}
    L_\lambda(x,y) = C_n \, (\sinh \rho)^{-\frac{n-2}{2}} P_{-\frac{1}{2}+i\lambda}^{-\frac{n-2}{2}}(\cosh \rho),
  \end{equation}
  where $\rho = d_{\hn}(x,y)$ and $P_\nu^\mu$ is the associated Legendre function of the first kind, solution of the equation
  \[(1-z^2) \partial^2_z w - 2z \partial_z w + \left[ \nu(\nu+1) - \frac{\mu^2}{1-z^2} \right] w = 0.\]
  Moreover, for $\rho>0, n\geq 1$ odd, there holds
  \begin{equation}\label{Lodd}
    L_\lambda(\rho) = C_n \left( -\frac{1}{\sinh \rho} \frac{d}{d\rho} \right)^{\frac{n-1}{2}} \cos(\lambda \rho),
  \end{equation}
  and for $n\geq 2$ even,
  \begin{equation}\label{Leven}
    L_\lambda(\rho) = C_n \int_{\rho}^\infty \frac{\sinh r}{\sqrt{\cosh r - \cosh \rho}} \left( -\frac{1}{\sinh r} \frac{d}{dr} \right)^{\frac{n}{2}} \cos(\lambda r) \, dr.
  \end{equation}
\end{lemma}
The kernel of Green's function $G_s(x,y)$ can be considered as the inverse of the fractional Laplacian. In \cite{BGS15}, the authors provided the information of $\mathcal{K}_s$. The compuation for $G_s$ is similar, and we present the detailed proof here for completeness.

\begin{proof}[Proof of Theorem \ref{thm-asymptotics}]

We assume that $n$ is odd; the calculations for $n$ even are similar. Using \eqref{Lodd} we have for $n\geq 3$ odd that
\begin{align*}
G_s(\rho)=&\left(\frac{\partial_\rho}{\sinh\rho}\right)^\frac{n-1}{2}\left( \int_{-\infty}^{\infty}
\left(\lambda^2+\tfrac{(n-1)^2}{4}\right)^{-s}\cos(\lambda\rho) \,d\lambda\right)\\
=&\left(\frac{\partial_\rho}{\sinh\rho}\right)^\frac{n-1}{2}\left( \int_{-\infty}^{\infty}
\left(\lambda^2+\tfrac{(n-1)^2}{4}\right)^{-s} e^{-i\lambda\rho}\, d\lambda\right).
\end{align*}
The last equality follows from $\left(\lambda^2+\tfrac{(n-1)^2}{4}\right)^{-s} \sin (\lambda\rho)$ being odd.

Notice that the equality above holds when $s>\frac{1}{2}$, since in this case $\left(\lambda^2+\tfrac{(n-1)^2}{4}\right)^{-s} \in L^1(\rn)$.
Hence for $s\in (0,\frac{1}{2})$, we need to compute the distributional Fourier transform of $h(\lambda):= \left(\lambda^2+\tfrac{(n-1)^2}{4}\right)^{-s}$. 
Notice
\begin{equation*}
  \left(\lambda^2+\tfrac{(n-1)^2}{4}\right) \partial_\lambda h=-2s\lambda h
\end{equation*}
in the distributional sense.
Taking Fourier transform, we have
\begin{equation*}
  \left(-\partial_{\rho\rho}-2\partial_\rho+\tfrac{(n-1)^2}{4}\right) (i\rho \,\hat h)=-2s(i\partial_\rho \hat h),
\end{equation*}
or equivalently
\begin{equation*}
   \rho\, \partial_{\rho\rho} \hat h+2(1-s)\,\partial_\rho \hat h- \tfrac{(n-1)^2}{4}\rho\, \hat h=0.
\end{equation*}
By the change of variables $r=\frac{n-1}{2}\rho$ and denoting  $\varphi(r)=\hat h(\rho)$, we obtain the ODE
$$ \partial_{rr} \varphi+\frac{2(1-s)}{r}\partial_r \varphi-  \varphi=0.$$

According to Lemma \ref{lemma-Bessel}, the solution can be written as
 $$\hat h(\rho)=\rho^{-\frac{1}{2}+s}\left(\alpha_s K_{-\frac{1}{2}+s}\left(\tfrac{n-1}{2}\rho\right)+\beta_s I_{-\frac{1}{2}+s}\left(\tfrac{n-1}{2}\rho\right)\right),$$
where $K_{-\frac{1}{2}+s},\; I_{-\frac{1}{2}+s},$ are the solutions to the  modified Bessel equation given in the same Lemma.

Since $h$ is a tempered distribution, $\hat{h}$ can at most have polynomial growth, hence, necessarily $\beta_s=0$. Besides, $\hat{h}$ is defined in the classical sense when $s>0$.
Recalling the asymptotic formulas for Bessel functions from Lemma \ref{lemma-Bessel}, we have that
\begin{align*}
  &\hat{h}\sim \rho^{-1+2s} \quad \text{as }\rho\to 0,\\
  &\hat{h}\sim \rho^{s-1}e^{-\frac{n-1}{2}\rho} \quad \text{as }\rho\to\infty.
\end{align*}
Finally, in order to obtain \eqref{G1} and \eqref{asymptotics-zero}, one notices that
\begin{align*}
  &\Big(\frac{\partial_\rho}{\sinh \rho}\Big)^m \sim \rho^{-m}\partial_\rho^m - (m-1)\rho^{m-1}\partial_\rho^{m-1}-(m-1)\rho^{-m}\partial_\rho^{m-2} \quad \text{as }\rho\to 0,\\
  &\Big(\frac{\partial_\rho}{\sinh \rho}\Big)^m \sim 2^m e^{-m\rho}\partial_\rho^m - (2^{m+1}e^{-m\rho}+2^m e^{-\rho})\partial_\rho^{m-1}, \qquad \text{as } \rho\to\infty.
\end{align*}

\end{proof}

Finally, we show the following monotonicity property of the kernel $G_{n,s}(\rho)$.
\begin{proposition}
  The kernel $G_{n,s}(\rho)$ is strictly decreasing with respect to $\rho$.
\end{proposition}

\begin{proof}
  We use the spectral representation of the Green kernel. Recall that
\[
G_s(\rho)
= \int_{-\infty}^{\infty}
\left(\lambda^2+\tfrac{(n-1)^2}{4}\right)^{-s}
L_\lambda(\rho)\,|c(\lambda)|^{-2}\,d\lambda,
\]
where $L_\lambda(\rho)$ is given in Lemma \ref{L_representation}. In particular,
$L_\lambda(\rho)$ depends only on $\rho=d_{\hn}(x,y)$ and is the spherical function associated with the eigenvalue
$-(\lambda^2+\tfrac{(n-1)^2}{4})$.

For radial function $f$, we have $\Delta_{\mathbb H^n} f(\rho)=f''(\rho)+(n-1)\coth(\rho)f'(\rho)$. Thus, 
\[
L_\lambda''(\rho)+(n-1)\coth(\rho)L_\lambda'(\rho) = -(\lambda^2+\tfrac{(n-1)^2}{4})L_\lambda(\rho).
\] 
Mulitplying both sides by $\sinh^{n-1}(\rho)$, we have
\[(\sinh^{n-1}(\rho)L_\lambda'(\rho))' = -(\lambda^2+\tfrac{(n-1)^2}{4})\sinh^{n-1}(\rho)L_\lambda(\rho)<0.\]
Thus, $\sinh^{n-1}(\rho)L_\lambda'(\rho)$ is strictly decreasing in $\rho$, which implies $L_\lambda'(\rho)<0$ and $G_s'(\rho)<0$.
\end{proof}

\begin{comment}
  
\begin{proof}
  If $n=2m+1$, we have
  \begin{align*}
    G_{2m+1,s}(\rho) = C_1 \left( -\frac{\partial_\rho}{\sinh \rho} \right)^{m} \left( \rho^{-\frac{1-2s}{2}} K_{s-\frac{1}{2}} \left( \frac{n-1}{2} \rho \right) \right).
  \end{align*}
  Thus,
  \begin{align*}
      &\frac{d}{d\rho}G_{2m+1,s}(\rho)=C_1(-\sinh\rho)\left( -\frac{\partial_\rho}{\sinh \rho} \right)^{m+1} \left( \rho^{-\frac{1-2s}{2}} K_{s-\frac{1}{2}} \left( \frac{n-1}{2} \rho \right) \right)\\
      &=C_1(-\sinh\rho)\cdot G_{2m+3,s}(\rho)<0.
  \end{align*}
  If $n=2m$, we have
  \begin{align*}
    G_{2m,s}(\rho) &= C_1 \int_{\rho}^{\infty} \frac{\sinh r}{\sqrt{\pi}\sqrt{\cosh r - \cosh\rho}} 
    \left( -\frac{\partial_r}{\sinh r} \right)^{m}
    \left( r^{-\frac{1-2s}{2}} K_{s-\frac{1}{2}} \left( \tfrac{n-1}{2}r \right) \right) dr.
  \end{align*}
  Now if we set $b=\sqrt{\cosh r-\cosh\rho}$, we get
  \begin{align*}
      \frac{d}{d\rho}G_{2m,s}(r)=\frac{1}{\sqrt{\pi}}\int_0^\infty \frac{d}{d\rho}G_{2m+1,s}(r(b,\rho)) db,
  \end{align*}
  where $r(b,\rho)$ is defined by $b=\sqrt{\cosh r-\cosh\rho}$.
  Since 
  \begin{equation*}
    \frac{d}{d\rho}G_{2m+1,s}(r)=\frac{dG_{2m+1,s}(r)}{dr}\frac{dr}{d\rho}=\frac{dG_{2m+1,s}(r)}{dr}\frac{\sinh\rho}{\sinh r}<0,
  \end{equation*}
  we have $\frac{d}{d\rho}G_{2m,s}(r)<0$. The proof is now concluded.
\end{proof}
\end{comment}

\section{Proof of Theorem \ref{thm1}}\label{sec:proof}

To start with, we show that the differential equation \eqref{main_equation_1} is equivalent to an integral equation.

\begin{lemma}\label{DE=IE}
Let $s\in(0,1)$ and $1<p\leq\frac{n+2s}{n-2s}$. Assume $u\ge 0$ and
\begin{equation}\label{assump_DEIE}
 u\in L^{\frac{2n}{n-2s}}(\hn).
\end{equation}
If $p<\frac{n+2s}{n-2s}$, assume in addition that
\begin{equation}\label{assump_DEIE_up}
 u^p\in L^{\frac{2n}{n+2s}}(\hn).
\end{equation}
Then if $u$ is a weak solution of
\begin{equation}\label{diff_eq}
  (-\Delta_{\hn})^s u = u^p \quad \text{in }\hn,
\end{equation}
then $u$ satisfies the integral equation
\begin{equation}\label{int_eq}
  u(x)=\int_{\hn} G_s\big(\rho(x,y)\big)\,u(y)^p\,dV_y\qquad\text{for a.e. }x\in\hn,
\end{equation}
where $G_s$ is the Green's function defined in \eqref{Green's fun}. The converse also holds, that is, if $u$ satisfies \eqref{int_eq}, then $u$ is a weak solution of \eqref{diff_eq}.
\end{lemma}

\begin{proof}
By Helgason-Fourier transform, we have for any $\varphi\in C^\infty_c(\hn)$:
\begin{align*}
  \widehat{(-\Delta_{\mathbb{H}^n})^s \varphi}(\lambda,\theta) &= \left(\lambda^2 + \tfrac{(n-1)^2}{4}\right)^s \hat{\varphi}(\lambda,\theta),
\end{align*}
and 
\begin{align}\label{test_laplacian}
  (-\Delta_{\mathbb H^n})^s\varphi(x)=\int_{-\infty}^\infty\int_{\mathbb S^{n-1}}
\Bigl(\lambda^2+\frac{(n-1)^2}{4}\Bigr)^s\widehat{\varphi}(\lambda,\theta)
\,\bar{h}_{\lambda,\theta}(x)
\,|c(\lambda)|^{-2}\,d\theta\,d\lambda.
\end{align}
Recall the definition of weak solution of \eqref{diff_eq}, we have
\begin{align*}
  \int_{\mathbb{H}^n} u(x) (-\Delta_{\mathbb{H}^n})^s \varphi(x) dV_x &= \int_{\mathbb{H}^n} u(x)^p \varphi(x) dx.
\end{align*}
Inserting \eqref{test_laplacian} to the left side, we obtain
\begin{align*}
  &\int_{\mathbb H^n} u(x)(-\Delta_{\mathbb H^n})^s\varphi(x)\,dx\\
&=
\int_{-\infty}^\infty \int_{\mathbb S^{n-1}}
\Bigl(\lambda^2+\tfrac{(n-1)^2}{4}\Bigr)^s
\widehat{\varphi}(\lambda,\theta)
\left[
\int_{\mathbb H^n}
u(x)\bar{h}_{\lambda,\theta}(x)\,dx
\right]
|c(\lambda)|^{-2}\,d\theta\,d\lambda.
\end{align*}

Using the inversion formula, the right side becomes
\begin{align*}
  \int_{\mathbb H^n} u(x)^p \varphi(x)\,dx
&=
\int_{\mathbb H^n} u(x)^p
\int_{-\infty}^\infty \int_{\mathbb S^{n-1}}
\widehat{\varphi}(\lambda,\theta)\bar{h}_{\lambda,\theta}(x)
|c(\lambda)|^{-2}\,d\theta\,d\lambda
\,dV_x\\
&=
\int_{-\infty}^\infty \int_{\mathbb S^{n-1}}
\widehat{\varphi}(\lambda,\theta)
\left[\int_{\mathbb H^n}
u(x)^p \bar{h}_{\lambda,\theta}(x)\,dV_x
\right]
|c(\lambda)|^{-2}\,d\theta\,d\lambda.
\end{align*}

Therefore, we have
\begin{align*}
\Bigl(\lambda^2+\tfrac{(n-1)^2}{4}\Bigr)^s
\int_{\mathbb H^n}
u(x)\bar{h}_{\lambda,\theta}(x)\,dV_x
=
\int_{\mathbb H^n}
u(x)^p \bar{h}_{\lambda,\theta}(x)\,dV_x
\end{align*}
Equivalently,
\[(\mu^2+\rho_0^2)^s\,\widehat u(\mu,\theta)=\widehat{u^p}(\mu,\theta).\]
Thus, by inversion formula again, we obtain
\begin{align*}
  u(x) &=  \int_{\mathbb{H}^n} u(y)^p \int_{-\infty}^\infty \int_{\mathbb{S}^{n-1}} \left(\lambda^2 + \tfrac{(n-1)^2}{4}\right)^{-s} \bar{h}_{\lambda,\theta}(x) h_{\lambda,\theta}(y)| c(\lambda)|^{-2}d\theta d\lambda dV_y\\
  &= \int_{\mathbb{H}^n} G_s(\rho(x,y)) u(y)^p dV_y.
\end{align*}

\end{proof}

\subsection{Subcritical case \texorpdfstring{$1<p<\frac{n+2s}{n-2s}$}{}}
\hfill

We first recall the following result in \cite{BP25}.
\begin{proposition}
  If $u\in H^{s}$ is a weak solution of \eqref{main_equation_int} with $1<p<\frac{n+2s}{n-2s}$, then $u\in L^\infty(\hn)$.
\end{proposition}

We then establish the decay property of $u$ at infinity.
\begin{lemma}\label{prop:decay_subcritical}
Let $u\ge0$ satisfy \eqref{main_equation_int} with subcritical $p$, then $u(x)\to 0$ as $\rho(x,o)\to\infty$.
\end{lemma}

\begin{proof}
Fix $\varepsilon>0$. Choose $\delta>0$, and $R>0$ large and split the integral into three parts:
\[
\hn= B_R(o) \cup (\mathbb H^n\setminus B_R(o))\cap B_\delta(x) \cup (\mathbb H^n\setminus B_R(o))\setminus B_\delta(x).\]
For the local part, we have that for any fixed $y\in B_R$, $\rho(x,y)\to\infty$ as $\rho(x,o)\to\infty$. By kernel estimate, $G_s(\rho(x,y))\lesssim \rho^{-1+s}e^{-(n-1)\rho(x,y)}$. Hence
\[
\int_{B_R(o)}G_s(\rho(x,y))u(y)^p\,dV_y \lesssim \rho(x,o)^{-1+s}e^{-(n-1)(\rho(x,o)-R)}\int_{B_R}u(y)^p\,dV_y < \varepsilon
\]
as $\rho(x,o)\to\infty$.
 
For the second part, we have
\begin{align*}
  \int_{(\mathbb H^n\setminus B_R(o))\cap B_\delta(x)}G_s(\rho(x,y))u(y)^p\,dV_y \leq \norm{u}_{L^\infty}^p \int_{B_\delta(x)} G_s(\rho(x,y))\,dV_y\\
  \lesssim \norm{u}_{L^\infty}^p \int_0^\delta \rho^{-n+2s}\sinh^{n-1}\rho\,d\rho \lesssim \norm{u}_{L^\infty}^p \int_0^\delta \rho^{2s-1}\,d\rho <\varepsilon.
\end{align*}

For the last part, using H\"older's inequality, we have
\begin{align*}
  &\int_{(\mathbb H^n\setminus B_R(o))\setminus B_\delta(x)}G_s(\rho(x,y))u(y)^p\,dV_y\lesssim \norm{G_s}_{L^{a'}}\norm{u^p}_{L^a}.
\end{align*}
where $a=\frac{2n}{n+2s}$ and $a'=\frac{2n}{n-2s}$. We see that
\[\|G_s(\rho(x,\cdot))\|_{L^{a'}}^{a'}\lesssim \omega_{n-1}\int_{\delta}^\infty \rho^{-a'(-1+s)}e^{-a'(n-1)\rho}\sinh^{n-1}\rho\,d\rho<\infty,\]
and $\|u^p\|_{L^a(\mathbb H^n\setminus B_R(o))}<\varepsilon$ for $R$ large enough. Therefore, $u(x) = \int_{\hn} G_s(x,y)u(y)^p\,dV_y \to 0$ as $\rho(x,o)\to\infty$.

\end{proof}

\subsubsection*{The moving plane argument}
\hfill

Recall the foliation structure of $\hn$ and the notations:
\[U=\hn\cap\{x_1=0\}, \; U_s=A_s(U),\; \Sigma_\lambda=\bigcup_{s<\lambda}U_s.\]
The reflection with respect to $\Sigma_\lambda$ is denoted as
\[x_\lambda=I_\lambda(x), \quad u_\lambda(x)=u(x_\lambda),\quad \lambda\in\mathbb{R}.\]
Without loss of generality, we may choose any direction to be the $x_1$ direction, and define that for $\lambda<x^0_1$,
\begin{equation}\label{def_w}
  w_\lambda(x) = u_\lambda(x)-u(x).
\end{equation}

\begin{comment}
  Then $w_\lambda$ satisfies the following equation
\[(-\Deltah)^sw_\lambda(x) = u(x_\lambda)^p - u(x)^p > pw_\lambda u(x)^p \text{ in } \Sigma^-_\lambda.\]
Denote 
\[c_\lambda(x) := -p u(x)^p,\]
then we have
\begin{equation*}
  (-\Delta)^s w_\lambda + c_\lambda(x)w_\lambda(x) \geq 0. 
\end{equation*}
\end{comment}

We first show that the moving plane procedure can be initiated from the infinity.

Define the negative set
\[
\Sigma_\lambda^- := \{ x\in\Sigma_\lambda : w_\lambda(x):=u_\lambda(x)-u(x)<0\}.
\]

We want to prove for $\lambda$ sufficiently negative, $\Sigma_\lambda^-=\emptyset$.
We first have 
\begin{align*}
&u(x)-u_\lambda(x)\\
&=\int_{\Sigma_\lambda} G_s(x,y)u^p(y)dV_y+\int_{\Sigma^c_\lambda} G_s(x,y)u^p(y)dV_y\\
&\quad -\int_{\Sigma_\lambda} G_s(x,y)u^p(y_\lambda)dV_y-\int_{\Sigma^c_\lambda} G_s(x,y)u^p(y_\lambda)dV_y\\
&=\int_{\Sigma_\lambda} G_s(x,y)u^p(y)dV_y+\int_{\Sigma_\lambda} G_s(x,y_\lambda)u^p(y_\lambda)dV_y\\
&\quad -\int_{\Sigma_\lambda} G_s(x,y)u^p(y_\lambda)dV_y-\int_{\Sigma_\lambda} G_s(x,y_\lambda)u^p(y)dV_y\\
&=\int_{\Sigma_\lambda} G_s(x,y)u^p(y)dV_y+\int_{\Sigma_\lambda} G_s(x_\lambda,y)u^p(y_\lambda)dV_y\\
&\quad -\int_{\Sigma_\lambda} G_s(x,y)u^p(y_\lambda)dV_y-\int_{\Sigma_\lambda} G_s(x_\lambda,y)u^p(y)dV_y\\
&=\int_{\Sigma_\lambda} \left(G_s(x,y)-G_s(x_\lambda,y)\right)\left(u^p(y)-u^p(y_\lambda)\right)dV_y.
\end{align*}
Hence, by the definition of $\Sigma_\lambda^-$ and the monotonicity of $G_s$, we have
\begin{align*}
    u(x)-u_\lambda(x)&\leq \int_{\Sigma_\lambda^-} G_s(x,y)\left(u^p(y)-u^p(y_\lambda)\right)dV_y,
\end{align*}
since $\rho(x,y)>\rho(x_\lambda,y)$ for $y\in\Sigma_\lambda$.
For a fixed radius $R>0$, and any $x,y\in\mathbb H^n$, we decompose
\[
G_s(x,y) = G_1(x,y) + G_2(x,y),
\]
where
\[
G_1(x,y):= G_s(x,y)\,\chi_{\{\rho(x,y)\le R\}},\qquad
G_2(x,y):= G_s(x,y)\,\chi_{\{\rho(x,y)> R\}}.
\]
Correspondingly,  we write
\[
w_\lambda(x) = \int_{\mathbb H^n} G_s(x,y)\,u(y)^p\,d\mu(y)
= I_1(x) + I_2(x),
\]
where
\[
I_i(x) := \int_{\mathbb H^n} G_s^i(x,y)\,u(y)^p\,d\mu(y),
\qquad i=\{1,2\}.
\]

For near part, we observe that in any fixed geodesic ball of radius $R$, the hyperbolic metric is equivalent to the Euclidean one. Hence the classical HLS inequality applies locally, and one obtains
\begin{align}\label{near_est}
\| I_1 \|_{L^q(\Sigma_\lambda^-)}
\le C_R \,\| u^p \chi_{\Sigma_\lambda^-} \|_{L^{q'}(\Sigma_\lambda)},
\end{align}
which holds for $1/q=1/q'-2s/n$.
For the far part, use the exponential decay of $G_s$. We have $G_s(x,y)\lesssim d^{-1+s}e^{-(n-1) \rho(x,y)}$. Then for $x\in\Sigma_\lambda$,
\[
| I_2(x) |
\le \int_{\rho(x,y)>R} G_s(x,y) u(y)^p\,d\mu(y).
\]
Using H\"older's inequality, we have for $1/r + 1/q' = 1$,
\begin{align*}
| I_2(x) | \leq \norm{G_s(x,\cdot)}_{L^{r}(\{\rho(x,y)>R\})}\norm{u^p}_{L^{q'}(\{\rho(x,y)>R\})}
\end{align*}
and 
\begin{align}\label{far_est}
\|I_2\|_{L^q(\Sigma_\lambda^-)}
&\le \Big(\int_{\Sigma_\lambda^-} \|G_s(x,\cdot)\|_{L^{r}(\mathbb H^n\setminus B_R)}^q\,dV_x\Big)^{1/q}
\;\cdot\; \|u^p\chi_{\Sigma_\lambda}\|_{L^{q'}(\Sigma_\lambda)}\nonumber\\
&\le \sup_{x\in\Sigma_\lambda^-} \|G_s(x,\cdot)\|_{L^{r}(\mathbb H^n\setminus B_R)}\;
\big(\mu(\Sigma_\lambda^-)\big)^{1/q}\ \|u^p\chi_{\Sigma_\lambda}\|_{L^{q'}(\Sigma_\lambda)}.
\end{align}

Combining \eqref{near_est} and \eqref{far_est} we obtain
\begin{align*}
\| w_\lambda \|_{L^q(\Sigma_\lambda^-)}
&\leq C_R \| u^p  \|_{L^{q'}(\chi_{\Sigma_\lambda^-})} + \varepsilon(R,\lambda)\,\big(\mu(\Sigma_\lambda^-)\big)^{1/q}\| u^p\|_{L^{q'}(\chi_{\Sigma_\lambda^-})}\\
& \lesssim \big(\mu(\Sigma_\lambda^-)\big)^{1/q}\| u^p\|_{L^{q'}(\chi_{\Sigma_\lambda^-})}.
\end{align*}
Then by the mean value theorem and H\"older's inequality again, there holds
\begin{align*}
    \| w_\lambda \|_{L^q(\Sigma_\lambda^-)} &\leq \big(\mu(\Sigma_\lambda^-)\big)^{1/q}\norm{u_\lambda^p - u^p}_{L^{q'}(\Sigma_\lambda^-)} \leq \big(\mu(\Sigma_\lambda^-)\big)^{1/q}\norm{p \xi^{p-1} (u_\lambda - u)}_{L^{q'}(\Sigma_\lambda^-)} \\
    &\leq \big(\mu(\Sigma_\lambda^-)\big)^{1/q} \norm{u^{p-1}}_{L^{n/2r}(\Sigma_\lambda^-)} \norm{u_\lambda - u}_{L^{q}(\Sigma_\lambda^-)}.
\end{align*}
This implies that for negative enough $\lambda$, $\|u-u_\lambda\|_{L^q(\Sigma_\lambda^-)}=0$, which means that $\Sigma_\lambda^-$ has measure zero.

Now we shift $U_\lambda$ inward as long as $ u_\lambda \geq u$ in $\Sigma_\lambda$. Suppose that there exists such $\bar{\lambda}$ that $u(x) > u_{\bar{\lambda}}(x)$ on $\Sigma_{\bar{\lambda}}$.
We deduce again by a standard compactness argument that there exist $\bar{\lambda}-\varepsilon<\lambda\leq \bar{\lambda}$ such that
$$\|u-u_\lambda\|_{L^q(\Sigma_\lambda)}\leq C\|f^\prime(\xi)\|_{L^{\frac{n}{2k}}(\Sigma_\lambda^-)}\|u-u_\lambda\|_{L^q(\Sigma_\lambda^-)}.$$
When $\varepsilon$ is small, $\Sigma_\lambda^-$ is close to zero so that $\|u_t-u\|_{L^q(\Sigma_\lambda^-)}=0$.
This implies that we can keep moving the plane $U_{\lambda_0}$. Now we see $u(x)\leq u(\bar{x})$ with respect to $\Sigma_{\lambda_0}$.
A similar argument shows $u(\bar{x})\geq u(x)$ by rotation. Consequently, there exists a point $P$ such that $u$ is constant on the geodesic spheres center at $P$.

\subsection{Critical case \texorpdfstring{$p=\frac{n+2s}{n-2s}$}{}}

\begin{proposition}\label{prop:decay_critical}
Let $u\ge0$ satisfy \eqref{main_equation_1} with critical $p$. If $u\in L^q(\mathbb H^n)$ for any $q>p$, then $u(x)\to 0$ as $\rho(x,o)\to\infty$.
\end{proposition}

\begin{proof}
Fix $\varepsilon>0$. Choose $R>0$ large and write
\[
u(x)=\int_{B_R(o)}G_s(\rho(x,y))u(y)^p\,dV_y + \int_{\mathbb H^n\setminus B_R(o)}G_s(\rho(x,y))u(y)^p\,dV_y.
\]
For the local part, we have that for any fixed $y\in B_R$, $\rho(x,y)\to\infty$ as $\rho(x,o)\to\infty$. By kernel estimate, $G_s(\rho(x,y))\lesssim \rho^{-1+s}e^{-(n-2s)\rho(x,o)}$. Hence
\[
\int_{B_R}G_s(\rho(x,y))u(y)^p\,dV_y \lesssim \rho^{-1+s}e^{-(n-2s)\rho(x,o)}\int_{B_R}u(y)^p\,dV_y \to 0
\]
uniformly in $x$ as $\rho(x,o)\to\infty$.
 
For the far away part, we choose $a\in(1,\frac{q}{p})$, and let $a'$ denote the conjugate of $a$. By H\"older's inequality,
\[
\left|\int_{\mathbb H^n\setminus B_R(o)}G_s(\rho(x,y))u(y)^p\,dV_y\right| \leq \|G_s(\rho(x,\cdot))\|_{L^{a'}(\mathbb H^n\setminus B_R)}\;\|u^p\|_{L^{a}(\mathbb H^n\setminus B_R)}.
\]
$\|u^p\|_{L^{a}(\mathbb H^n\setminus B_R)}$ is clearly finite since $ap<q$.  
On the other hand, for large $R$, we have 
\[
\int_{ \mathbb H^n\setminus B_R} G_s(\rho(x,y))^{a'}\,dV_y
\lesssim \int_R^\infty \rho^{-1+s}e^{-a'(n-2s)\rho} e^{(n-1)\rho}\,d\rho
= \int_R^\infty \rho^{-1+s}e^{-(a'(n-2s)-(n-1))\rho}\,d\rho
\]
We are able to choose $a$ sufficiently close to $1$, so that $a'(n-2s)-(n-1)>0$. Thus, $\|G_s(\rho(x,\cdot))\|_{L^{a'}(\mathbb H^n\setminus B_R)}$ can be made arbitrarily small for large enough $R$. The proof is now complete.
\end{proof}
From the above estimate at infinity, the moving plane argument in the subcritical case applies here as well, and we conclude that $u$ is radially symmetric about some point in $\hn$.

\section{Several Maximum Principles on \texorpdfstring{$\hn$}{}}\label{sec:MP}
Chen, Li and Li \cite{CLL} established various maximum principles involving the fractional Laplacian on $\rn$, see also their book \cite{CLM_book}. These are essential in their moving plane argument when one deals with the nonlocal operator directly. Here we extend their results to the hyperbolic space $\hn$, which may be of independent interest.

\begin{theorem}[Maximum Principle for Fractional Laplacian on $\hn$]\label{MP}
  Let $\Omega$ be a bounded domain in $\hn$. Assume that $u \in L^\alpha \cap C^{1,1}_{loc}(\Omega)$ and is lower semi-continuous on $\bar{\Omega}$ . If
  \begin{equation*}
    \begin{cases}
      (-\Delta_{\hn})^s u \geq 0, &\text{ in } \Omega,\\
      u \geq 0, &\text{ in } \hn\setminus\Omega,
    \end{cases}
  \end{equation*}
  where $c(x)$ is bounded from below in $\Omega$. Then 
  \begin{equation}\label{mp_u}
    u \geq 0 \text{ in } \Omega.
  \end{equation}
  If $u = 0$ at some point in $\Omega$, then $u(x) = 0$ almost everywhere in $\hn$.
\end{theorem}

\begin{proof}
  If \eqref{mp_u} does not hold, then the lower semi-continuity of $u$ indicates that there exists a $x^0 \in \bar{\Omega}$ such that  $u(x^0) = \min\limits_\Omega u < 0$.  And one can further deduce that $x^0$ is in the interior of $\Omega$.
  Then 
  \begin{align*}
    (-\Delta_{\hn})^s u(x^0) &= C_{n,s} \, \text{P.V.} \int_{\mathbb{H}^n} (u(x^0) - u(\xi)) \, \mathcal{K}_{n,s}(d(x,\xi)) \, d\xi\\
    &\leq C_{n,s} \int_{\mathbb{H}^n\setminus\Omega} (u(x^0) - u(\xi)) \, \mathcal{K}_{n,s}(d(x,\xi)) \, d\xi\\
    &\leq C_{n,s} \int_{\mathbb{H}^n\setminus\Omega} -u(\xi) \, \mathcal{K}_{n,s}(d(x,\xi)) \, d\xi < 0.
  \end{align*}
If at some point $x^0 \in \Omega, u(x^0) = 0$, then 
\begin{align*}
  (-\Delta_{\hn})^s u(x^0) &= C_{n,s} \, \text{P.V.} \int_{\mathbb{H}^n} - u(\xi) \, \mathcal{K}_{n,s}(d(x,\xi)) \, d\xi \geq 0.
\end{align*}
Since $u\geq 0$ in both $\Omega$ and $\hn\setminus\Omega$, we have $u(\xi) = 0$ almost everywhere in $\hn$.
\end{proof}

Then we introduce a maximum principle for anti-symmetric functions.
\begin{theorem}
\label{MP_anti}
  Let $\Omega$ be a bounded domain in $\hn$. Assume that $u \in L^\alpha \cap C^{1,1}_{loc}(\Omega)$ and is lower semi-continuous on $\bar{\Omega}$ . If
  \begin{equation*}
    \begin{cases}
      (-\Delta_{\hn})^s u + c(x)u \geq 0, &\text{ in } \Omega,\\
      u \geq 0, &\text{ in } \Sigma,\\
      u(x) = -u(x^\lambda), &\text{ in } \Sigma,
    \end{cases}
  \end{equation*}
  where $c(x)$ is bounded from below in $\Omega$. Then 
  \begin{equation}\label{mp_anti_u}
    u \geq 0 \text{ in } \Omega.
  \end{equation}
  If $u = 0$ at some point in $\Omega$, then $u(x) = 0$ almost everywhere in $\hn$.
\end{theorem}
The next result is a maximum principle in narrow regions.
\begin{theorem}
\label{narrow_region}
  Let $U_\lambda,\, \Sigma_\lambda, x^\lambda$ be defined as in Section \ref{sec:prelim}. Assume that $u \in L^\alpha \cap C^{1,1}_{loc}(\Omega)$ and is lower semi-continuous on $\bar{\Omega}$ . If
  Let $\Omega$ be an bounded narrow region in $\Sigma_\lambda$, which is contained in $\{x\mid \lambda-l<x_1<\lambda\}$ with small $l$. Assume that $u \in L^\alpha \cap C^{1,1}_{loc}(\Omega)$ and is lower semi-continuous on $\bar{\Omega}$, which satisfies
  \begin{equation*}
    \begin{cases}
      (-\Delta_{\hn})^s u + c(x)u \geq 0, &\text{ in } \Omega,\\
      u \geq 0, &\text{ in } \Sigma_\lambda\setminus\Omega,\\
      u(\tilde{x}) = -u(x^\lambda), &\text{ in } \Sigma_\lambda.
    \end{cases}
  \end{equation*}
  If $c(x)$ is bounded from below in $\Omega$, then 
  \begin{equation}\label{narrow_u}
    u \geq 0 \text{ in } \Omega.
  \end{equation}
  Furthermore, if $u = 0$ at some point in $\Omega$, then $u(x) = 0$ almost everywhere in $\hn$.
\end{theorem}
\begin{proof}
  If \eqref{narrow_u} does not hold, then the lower semi-continuity of $u$ indicates that there exists a $x^0 \in \bar{\Omega}$ such that  $u(x^0) = \min\limits_\Omega u < 0$.  And one can further deduce that $x^0$ is in the interior of $\Omega$.
  Then 
  \begin{align*}
    &(-\Delta_{\hn})^s u(x^0) = C_{n,s} \, \text{P.V.} \int_{\mathbb{H}^n} (u(x^0) - u(\xi)) \, \mathcal{K}_{n,s}(d(x,\xi)) \, d\xi\\
    &= C_{n,s} \text{P.V.}\left\{\int_{\Sigma_\lambda} (u(x^0) - u(\xi)) \, \mathcal{K}_{n,s}(d(x,\xi)) \, d\xi +  \int_{\Sigma^c_\lambda} (u(x^0) - u(\xi)) \, \mathcal{K}_{n,s}(d(x,\xi)) \, d\xi\right\}\\
    &= C_{n,s} \text{P.V.}\left\{\int_{\Sigma_\lambda} (u(x^0) - u(\xi)) \, \mathcal{K}_{n,s}(d(x,\xi)) \, d\xi +  \int_{\Sigma_\lambda} (u(x^0) + u(\xi)) \, \mathcal{K}_{n,s}(d(x,\xi^\lambda)) \, d\xi\right\}\\
    &\leq C_{n,s} \text{P.V.}\left\{\int_{\Sigma_\lambda} (u(x^0) - u(\xi)) \, \mathcal{K}_{n,s}(d(x,\xi^\lambda)) \, d\xi +  \int_{\Sigma_\lambda} (u(x^0) + u(\xi)) \, \mathcal{K}_{n,s}(d(x,\xi^\lambda)) \, d\xi\right\}\\
    &= 2C_{n,s} u(x^0) \int_{\Sigma_\lambda} \mathcal{K}_{n,s}(d(x,\xi^\lambda)) \, d\xi.
  \end{align*}
  To estimate the integral above, given that $x^0$ lies in a narrow region close to the hyperplane $\{y\mid y_1=\lambda\}$. Consider the extreme case that $x^0$ lies on the hyperplane $\{y\mid y_1=\lambda\}$, then 
  \begin{align*}
    \int_{\Sigma_\lambda} \mathcal{K}_{n,s}(d(x,\xi^\lambda)) \, d\xi &\sim \int_{\Sigma_\lambda}\rho^{-n-2s} \, d\xi=\infty,
  \end{align*}
  where $\rho=\disth(\xi,x^0)$. Then by integrating on a domain that is suﬃciently close to the hyperplane $P$, we can obtain a lower bound of the integral. Since $c(x)$ is bounded from below, we can choose $l$ small enough such that $(-\Delta_{\hn})^s u(x^0) + c(x^0)u(x^0) < 0$, which contradicts the assumption.
\end{proof}

\begin{theorem}\label{decay_infinity}
  Let $U_\lambda,\, \Sigma_\lambda, x^\lambda$ be defined as in Section \ref{sec:prelim}, and let $\Omega$ be an unbounded domain in $\Sigma_\lambda$. Assume that $u \in L^\alpha \cap C^{1,1}_{loc}(\Omega)$ and is lower semi-continuous on $\bar{\Omega}$ . If
  \begin{equation*}
    \begin{cases}
      (-\Delta_{\hn})^s u + c(x)u \geq 0, &\text{ in } \Omega,\\
      u \geq 0, &\text{ in } \Sigma_\lambda\setminus\Omega,\\
      u(\tilde{x}) = -u(x), &\text{ in } \Sigma_\lambda,
    \end{cases}
  \end{equation*}
  where $c(x)$ satisfies
  \begin{equation}\label{c_cond}
    \liminf\limits_{\rho(x,0)\to\infty}(d(x,0))^{(n+2s)}e^{c(n-1)d(x,0)}c(x)\geq 0,
  \end{equation}
  for some $c=c(n)>0$.
  Then there exists a constant $R_0>0$ depending on $c(x)$, such that if $u$ attains its negative minimum at $x^0\in\Omega$, then 
  \begin{equation*}
    \disth(x^0,0)\leq R_0.
  \end{equation*} 
\end{theorem}

\begin{proof}
  By definition, we have
  \begin{align*}
    &(-\Delta_{\hn})^s u(x^0) = C_{n,s} \, \text{P.V.} \int_{\mathbb{H}^n} (u(x^0) - u(\xi)) \, \mathcal{K}_{n,s}(d(x,\xi)) \, d\xi\\
    &= C_{n,s} \text{P.V.}\left\{\int_{\Sigma_\lambda} (u(x^0) - u(\xi)) \, \mathcal{K}_{n,s}(d(x,\xi)) \, d\xi +  \int_{\Sigma^c_\lambda} (u(x^0) - u(\xi)) \, \mathcal{K}_{n,s}(d(x,\xi)) \, d\xi\right\}\\
    &= C_{n,s} \text{P.V.}\left\{\int_{\Sigma_\lambda} (u(x^0) - u(\xi)) \, \mathcal{K}_{n,s}(d(x,\xi)) \, d\xi +  \int_{\Sigma_\lambda} (u(x^0) + u(\xi)) \, \mathcal{K}_{n,s}(d(x,\xi^\lambda)) \, d\xi\right\}\\
    &\leq C_{n,s} \text{P.V.}\left\{\int_{\Sigma_\lambda} (u(x^0) - u(\xi)) \, \mathcal{K}_{n,s}(d(x,\xi^\lambda)) \, d\xi +  \int_{\Sigma_\lambda} (u(x^0) + u(\xi)) \, \mathcal{K}_{n,s}(d(x,\xi^\lambda)) \, d\xi\right\}\\
    &= 2C_{n,s} u(x^0) \int_{\Sigma_\lambda} \mathcal{K}_{n,s}(d(x,\xi^\lambda)) \, d\xi.
  \end{align*}
  For each fixed $\lambda$, when $\rho(x^0,0) \geq \lambda$, we have $B_{\rho(x^0,0)} (x^1)\subset \Sigma_\lambda^c$, with $x^1 = (3\rho(x^0,0)+x_1^0, (x^0)^\prime)$. Then 
  \begin{align*}
    \int_{\Sigma_\lambda} \mathcal{K}_{n,s}(d(x,\xi^\lambda)) \, d\xi &\geq \int_{B_{\rho(x^0,0)} (x^1)} \mathcal{K}_{n,s}(d(x,\xi)) \, d\xi \\
    &\geq \int_{B_{\rho(x^0,0)} (x^1)} \mathcal{K}_{n,s}(c_1\cdot d(x^0,0)) \, d\xi\\
    &\gtrsim \mathcal{K}_{n,s}(c\cdot d(x^0,0))\textit{vol}(B_{\rho(x^0,0)}),\\
  \end{align*}
  where $c_1>0$ is a finite constant determined by the hyperbolic geometry. 
  Then we see that 
  \begin{align*}
    0\leq (-\Delta_{\hn})^s u(x^0) + c(x^0)u(x^0) &\lesssim  u(x^0) \left(\mathcal{K}_{n,s}(c\cdot d(x^0,0)) + c(x^0)\right).
  \end{align*}
  Letting $d(x^0,0) \to \infty$ and by the asymptotic behavior of the kernel $\mathcal{K}_{n,s}$, we have
  \begin{align*}
    \mathcal{K}_{n,s}(d(x,0)) \to c_2(d(x,0))^{-(n+2s)}e^{-c_1(n-1)d(x,0)}.
  \end{align*}
  which is a contradiction with the assumption \eqref{c_cond} on $c(x)$ since $u(x^0)<0$.
\end{proof}

\section*{Acknowlegment} The author would like to thank Professor Bruno, T and Professor Papageorgiou, E for their valuable comments and for informing some crucial recent results in this topic.

\end{document}